\documentclass[12pt]{amsart}
\usepackage{amssymb}
\usepackage{amsmath,amssymb,amsfonts,amsthm,graphics,
latexsym, amscd, amsfonts, epsfig, eepic,epic,psfrag,setspace,eufrak}
\usepackage{mathrsfs}
\usepackage{url}
\usepackage{color}
\usepackage{eucal}
\usepackage{eufrak}
\usepackage[all]{xypic}
\usepackage{xspace}
\usepackage{natbib}

\theoremstyle{plain}
\newtheorem{theorem}{Theorem}
\newtheorem{corollary}{Corollary}

\newtheorem{lemma}{Lemma}

\theoremstyle{definition}
\newtheorem{definition}{Definition}

\theoremstyle{remark}

\numberwithin{equation}{section}

\begin{document}
\doublespacing

\title{Topology of RNA-RNA interaction structures}
\author{J{\o}rgen E. Andersen$^{1}$, Fenix W.D. Huang$^{2}$,  Robert C. Penner$^{3, 4}$ and
Christian M. Reidys$^{2 \star}$}

\maketitle

\begin{center}
$1$ Aarhus University, DK-8000 {\AA}rhus C, Denmark \\
$2$ University of Southern Denmark, Campusvej 55, DK-5230 Odense M, Denmark \\
$3$ Center for Quantum Geometry of Moduli Spaces Aarhus University, DK-8000 {\AA}rhus C, Denmark \\
$4$ Math and Physics Departments, California Institute of Technology, Pasadena, California, USA \\
$\star$ Correspondent author: phone: +45 24409251 \\ duck@santafe.edu
\end{center}

\centerline{\bf Abstract}
\qquad \: 
The topological filtration of interacting RNA complexes is studied
and the role is analyzed of certain diagrams called irreducible
shadows, which form suitable building blocks for more general
structures.
We prove that for two interacting RNAs, called interaction structures,
there exist for fixed genus only finitely many irreducible shadows.
This implies that for fixed genus there are only finitely many classes
of interaction structures.
In particular the simplest case of genus zero already
provides the formalism for certain types of structures that occur
in nature and are not covered by other filtrations.
This case of genus zero interaction structures is already of
practical interest, is studied here in detail and found to be
expressed by a multiple context-free grammar extending the usual
one for RNA secondary structures.
We show that in $O(n^6)$ time and $O(n^4)$ space complexity,
this grammar for genus zero interaction structures
provides not only minimum free energy solutions but also
the complete partition function and base pairing probabilities.


\section{Introduction}\label{S:Introduction}

RNA-RNA interactions constitute one of the fundamental mechanisms of
cellular regulation. For instance, small RNAs binding a larger (m)RNA
target include: the regulation of translation in both prokaryotes
\citet{Vogel:07} and eukaryotes \citet{McManus,Banerjee}, the targeting
of chemical modifications \citet{Bachellerie}, insertion editing
\citet{Benne} and transcriptional control \citet{Kugel}.
For a variety of RNA classes including miRNAs, siRNAs, snRNAs, gRNAs, and
snoRNAs, a salient feature is the formation of RNA-RNA interaction structures
that are far more complex than simple sense-antisense interactions.
Accordingly, the ability to predict the details of RNA-RNA interactions
in terms of the thermodynamics of binding and in its structural consequences
is a necessary prerequisite to understanding RNA-based regulation mechanisms.
The exact location of the binding and the subsequent impact of the interaction
on the structure of the target molecule has potentially profound biological
consequences.
In case of sRNA-mRNA interactions, such details determine whether
the sRNA is a positive or negative regulator of transcription depending on
whether binding exposes or covers the Shine-Dalgarno sequence
\citet{Sharma:07,Majdalani:02}. Effects along these lines have been observed
also using artificially designed opener and closer RNAs that regulate the
binding of the \emph{HuR} protein to human mRNAs
\citet{Meisner:04a,Hackermueller:05a}.

An RNA molecule is a linearly oriented sequence of four types of nucleotides, namely,
{\bf A}, {\bf U}, {\bf C}, and {\bf G}. This sequence is endowed with a
well-defined orientation from the $5'$- to the $3'$-end and referred to as
the backbone.
Each nucleotide can form a base pair by interacting with at most one other
nucleotide by establishing hydrogen bonds. Here we restrict ourselves to
Watson-Crick base pairs {\bf GC} and {\bf AU} as well as the wobble base
pairs {\bf GU}. In the following, base triples as well as other types of more
complex interactions are neglected. RNA structures can be presented as
diagrams by drawing the backbone horizontally and all base pairs as arcs
in the upper halfplane; see Figure~\ref{F:RNAp}.
This set of arcs provides our coarse-grained RNA structure in
particular ignoring any spatial embedding or geometry of the molecule
beyond its base pairs.
\begin{figure}[ht]
\centerline{\epsfig{file=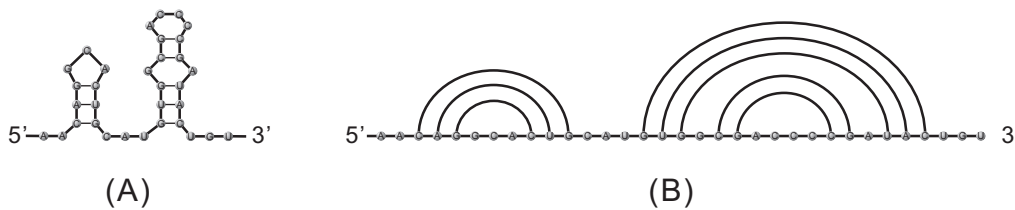,width=0.9\textwidth} \hskip8pt}
\caption{\small (A) An RNA secondary structure and (B) its diagram
representation.
}\label{F:RNAp}
\end{figure}
Accordingly, particular classes of base pairs translate into specific
structure categories, the most prominent of which are secondary
structures \citet{Kleitman:70,Nussinov:1978,Waterman:78,Waterman:79a}. When represented as diagrams,
secondary structures have only non-crossing base
pairs (arcs).
Beyond RNA secondary structures are the RNA pseudoknot structures that
allow for cross serial interactions \citet{Rivas}. There are several
meaningful filtrations of cross-serial interactions
\citet{Orland:02,Reidys:11a,Reidys:10w}.
Given an RNA coarse-grained structure class together with an energy
function, ``folding'' an RNA sequence means to compute a
minimum\footnote{with respect to the {\it a priori} specified energy
function} free energy configuration (MFE) or a partition function for
the sequence.

RNA interaction structures are structures over two backbones. We
distinguish internal arcs and external arcs as having their endpoints on
the same and different backbones, respectively.
Interaction structures are represented as two backbones with internal
and external arcs drawn in the upper halfplane. Alternatively, they can
be represented by drawing the two backbones on top of each other, see
Figure~\ref{F:diag_represent}.
\begin{figure}[ht]
\centerline{\epsfig{file=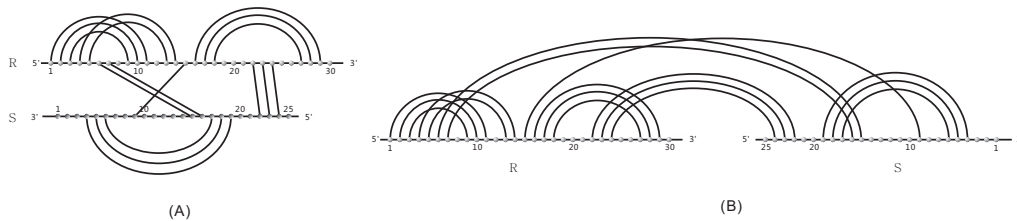,width=0.9\textwidth} \hskip8pt}
\caption{\small (A) Diagram representation of an RNA-RNA interaction structure.
(B) The representation of (A) with the two backbones drawn on a horizontal line.
}\label{F:diag_represent}
\end{figure}

The simplest approach for folding RNA-RNA interaction structures concatenates
two (or more) interacting sequences one after another remembering the
specific merge point (cut-point) and then employs the standard secondary
structure folding algorithm on a single strand with a slightly modified
energy model that treats loops containing cut-points as external elements.
The software tools
\texttt{RNAcofold} \citet{Hofacker,Bernhart},
\texttt{pairfold} \citet{Andronescu} and \texttt{NUPACK} \citet{Dirks:07}
subscribe to this strategy. This approach falls short predicting many
important motifs such as kissing-hairpin loops. The paradigm of concatenation
has also been generalized to include cross-serial interactions \citet{Rivas}.
The resulting model, however, still does not generate all relevant interaction
structures \citet{Backofen,Reidys:frame}.
An alternative line of thought, implemented in \texttt{RNAduplex} and
\texttt{RNAhybrid} \citet{rehmsmeier:04}, is to neglect all internal
base pairings in either strand, i.e., to compute the minimum free
energy (MFE) secondary structure of hybridization of otherwise
unstructured RNAs. \texttt{RNAup}
\citet{Mueckstein:05a,Mueckstein:08a} and \texttt{intaRNA}
\citet{Busch:08} restrict interactions to a single interval that
remains unpaired in the secondary structure for each partner. As a
special case, snoRNA/target complexes are treated more efficiently
using a specialized tool \citet{Tafer:09x} due to the highly
conserved interaction motif. Algorithmically, the approaches
mentioned so far are close relatives of the ``classical''
RNA folding recursions given by \citet{Zuker:84,Waterman:78}.
A different approach was taken independently by \citet{Pervouchine:04}
and \citet{Alkan:06}, who proposed MFE folding algorithms for predicting
the \emph{AP-structure} of two interacting RNA molecules. In this
model, the intramolecular structures of each partner are pseudoknot-free,
the intermolecular binding pairs are non-crossing, and there is no
so-called ``zig-zag'' motif, see Sec.~\ref{S:facts}.
The optimal joint structure can be computed in $O(N^6)$
time and $O(N^4)$ space by means of dynamic programming.

In contrast to the RNA secondary folding problem, where minimum energy
folding and partition functions can be obtained by  similar algorithms,
the case of interaction structures is more involved. The reason is that
simple unambiguous grammars are known for RNA secondary structures
\citet{Dowell:04} while the disambiguation of grammar underlying the
Alkan-Pervouchine algorithm requires the introduction of a large number
of additional non-terminals (which algorithmically translate into
additional dynamic programming tables).
The partition function was derived independently by \citet{Backofen}
(\texttt{piRNA}) and \citet{rip:09} (\texttt{rip1}). In \citet{Huang:10a},
probabilities of interaction regions as well as entire hybrid blocks
were derived. Although the partition function of joint structures can
also be computed in $O(N^6)$ time and $O(N^4)$ space, the current
implementations require  large computational resources.
\citet{Backofen:fast} recently achieved a substantial
speed-up making use of the observation that the external
interactions mostly occur between pairs of unpaired regions of
single structures. \citet{Chitsaz:09}, on the other hand, use
tree-structured Markov Random Fields to approximate the joint
probability distribution of multiple $(\geq 3)$ contact regions.
The RNA-RNA interaction structures of
\citet{Huang:10a,Alkan:06,Hofacker,Bernhart} have the following features:
\begin{itemize}
\item when drawing the two backbones on top of each other,
      all base pairs are non-crossing, i.e., no pseudoknots
      formed by internal or external arcs are allowed,
\item zig-zag motifs are disallowed.
\end{itemize}

This paper will relax the above constraints and propose a novel
filtration of RNA-RNA interaction structures based on the topological
fitration of RNA interaction structures.
Interaction structures that do not belong to the Alkan-Pervouchine class
exist: for instance the integral RNA (hTER) of the human telomerase
ribonucleoprotein has a conserved secondary structure that contains a
potential pseudoknot \citet{trans-pseudo}.
There is evidence that the two conserved
complementary sequences of one stem of the hTER pseudoknot domain
can pair intermolecularly in vitro, and that formation of this stem
as part of a novel ``transpseudoknot'' is required for the telomerase
to be active in its dimeric form.
The classification and expansion of pseudoknotted RNA structures
over one backbone via topological genus of the associated fatgraph
were first proposed by \citet{Orland:02,Penner:03,Bon:08}

In \citet{Reidys:11a,Zagier:95}, it was proved that for \emph{any} genus, there are only
\emph{finitely} many shadows, i.e., particular, simple atomic motifs.
In case of genus one, these shadows were first presented in \citet{Bon:08}.
Shadows give rise to a novel structure class, naturally generalizing
RNA secondary structures. These $\gamma$-structures \citet{Reidys:11a} are
generated by concatenation and nesting of irreducible building blocks
of genus $\le\gamma$.
We shall present the topological classification of RNA-RNA interaction
structures. This filtration gives rise to the notion of
$\gamma$-structures over two backbones.
In analogy to their one-backbone counterparts, $\gamma$-structures
over two backbones are composed of irreducible building blocks of genus
$\le \gamma$ and have accordingly arbitrarily high genus.
We shall see that for any fixed genus, there are only finitely
many irreducible shadows over two backbones.
In particular, we study genus zero structures over two backbones. The latter
are the two backbone analogue of RNA secondary structures\footnote{which are
well-known to be genus zero structures over one backbone}. $0$-structures
over two backbones already exhibit interesting features not shared with
AP-structures, see Figure~\ref{F:chy}.
\begin{figure}[ht]
\centerline{\epsfig{file=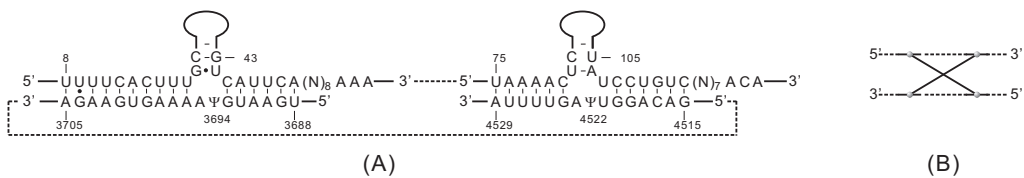,width=0.9\textwidth} \hskip8pt}
\caption{\small (A) Homo sapiens ACA27 snoRNA.
This H/ACA box RNA was cloned \citet{Kiss:04,Ofengand:97} from a HeLa
cell extract immunoprecipitated with an anti-GAR1 antibody. (B) The structure
contains two crossing hybrids, which cannot be found in AP-structures.
}\label{F:chy}
\end{figure}
We furthermore derive an unambiguous grammar for $0$-structures over
two backbones, which translates into an efficient dynamic programming algorithm.
This grammar, illustrated in Figure~\ref{F:outline}, allows the calculation of the minimum free
energy, partition function and Boltzmann-sampling. It explicitly treats
hybrids and gap structures, i.e., maximal regions with exclusively
intermolecular interactions and maximal regions with base pairs over one
backbone. The grammar thus facilitates the computation of the probability of
hybrids, the target interaction probability between two RNA strands, and the
probability of gap structures.
\begin{figure}[ht]
\centerline{\epsfig{file=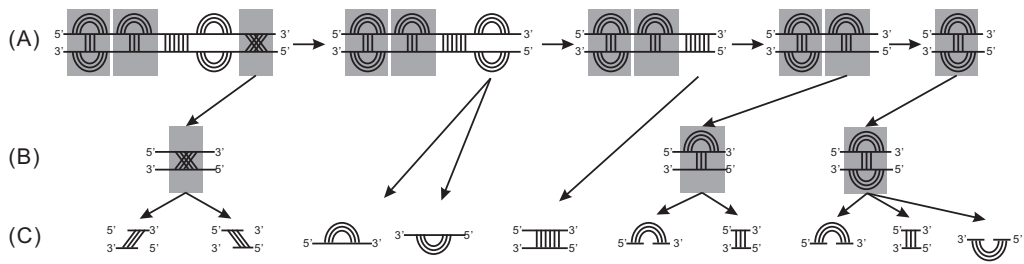,width=0.9\textwidth} \hskip8pt}
\caption{\small An unambiguous grammar of RNA-RNA interaction structures of
genus zero over two backbones. Basic building blocks are: tight structures
(gray), secondary structures and hybrid structures (A). Only tight structures
exhibit cross-serial interactions (B) and are further decomposed (C).
}\label{F:outline}
\end{figure}

\section{Basic facts}\label{S:facts}

\subsection{Diagrams}

A diagram is a labeled graph over the vertex set $[n]=\{1, \dots, n\}$ in
which each vertex has degree $\le 3$, represented by drawing its vertices
in a horizontal line and its arcs $(i,j)$, where $i<j$, in the upper
half-plane.
A backbone is a sequence of consecutive integers contained in $[n]$.
A diagram over $b$ backbones is a diagram together with a partition
of $[n]$ into $b$ backbones, see Figure~\ref{F:RNAp} (B).
In the following we shall denote the set of diagrams over one and two
backbones by $\mathbb{D}$ and $\mathbb{E}$ respectively.

The vertices and arcs of a diagram correspond to nucleotides and base pairs,
respectively.
For a diagram over $b$ backbones, the leftmost vertex of each backbone
denotes the $5'$ end of the RNA sequence, while the rightmost vertex
denotes the $3'$ end.
In case of $b>1$, we shall distinguish two types of arcs: an arc is called
{exterior} if it connects different backbones and {interior}
otherwise. Diagrams over $b$ backbones without exterior arcs are
disjoint unions of diagrams over one backbone.

The particular case $b=2$ is referred to as RNA interaction structures
\citet{rip:09,Huang:10a}, see Figure~\ref{F:diag_represent} (A).
As mentioned before, interaction structures are oftentimes represented
alternatively by drawing the two backbones $R$ and $S$ on top of each
other, indexing the vertices $R_1$ to be the $5'$ end of $R$ and $S_1$
to be the $3'$ of $S$.

A zig-zag is defined as follows:
given two sequences $R$ and $S$, suppose that  $R_aS_b$, (i.e., $R_a$ is base paired with $S_b$), $R_iR_j$, and $S_{i'}S_{j'}$
with
$i<a<j$ and $i'<b<j'$.  We say that $R_iR_j$ is subsumed in
$S_{i'}S_{j'}$, if for any $R_{k}S_{k'}\in I$, $i<k<j$ implies
$i'<k'<j'$.  Finally, a zigzag is a subgraph containing two dependent
interior arcs $R_{i_1}R_{j_1}$ and $S_{i_2}S_{j_2}$ neither one
subsuming the other, see Figure~\ref{F:zigzag}, where dependence here means that
there exists at least one exterior arc $R_{h}S_{\ell}$ such that
$i_1<h<j_1$ and $i_2<\ell<j_2$.

\begin{figure}[ht]
\centerline{\epsfig{file=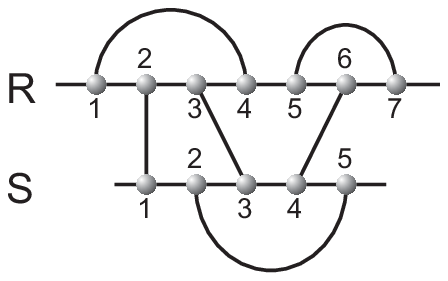,width=0.33\textwidth} \hskip8pt}
\caption{\small A zig-zag structure. $R_1R_4$ and $S_2S_5$ are dependent
interior arcs owing to  the base pair
 $R_3S_3$, but in view of $R_2S_1$ and $R_6S_4$, neither
subsumes the other.
}\label{F:zigzag}
\end{figure}


\subsection{From diagrams to topological surfaces}


One approach for deriving meaningful filtrations of RNA structure is
to pass from diagrams to topological surfaces \citet{Massey:69}.
It is natural to make this transition from combinatorics to topology
via fatgraphs \citet{protein,Penner:11}. A fatgraph $\mathbb{G}$, sometimes also called
``ribbon graph'' or ``map'', is a graph $G$ together with a collection
of cyclic orderings, called a fattening, one such ordering on the half-edges
incident on each vertex.
Each fatgraph $\mathbb{G}$ determines an oriented surface $F(\mathbb{G})$
as follows: let $V(G)$ be the set of $G$-vertices and $E(G)$ be the set
of $G$-edges. For each $v\in V(G)$, consider an oriented surface
isomorphic to a polygon $P_v$ with $2k$ sides containing $v$ in its interior where
$k$ is the valence of $v$. The incident edges of $v$ are also incident to
a univalent vertex contained in alternating sides of $P_v$, which are
identified with the incident half-edges in the natural way so that the
induced counter-clockwise cyclic ordering on the boundary of $P_v$ agrees
with the fattening of $\mathbb{G}$ about $v$.
The surface $F(\mathbb{G})$ is the quotient of the disjoint union
$\sqcup_{v\in V(G)}P_v$, where the frontier edges, which are oriented with
the polygons on their left, are identified by an orientation-reversing
homeomorphism if the corresponding half-edges lie in a common edge of $G$.
This defines the oriented surface $F(\mathbb{G})$, which is connected if and only if
$G$ is and is uniquely determined in this case by its genus $g=g(G)\ge 0$ and number $r=r(G)\ge 1$
of boundary components.
Since $F(\mathbb{G})$ contains $G$ as a deformation
retract, they share the Euler characteristic $v-e$, and the genus  of $F(\mathbb{G})$ is given by $2-2g-r=v-e$.

For an RNA diagram, we may draw a representation as usual so that the backbone is a horizontal line oriented  from
left to right, and the arcs lie in the upper half-plane.  This determines a unique fattening on any diagram, cf.\  the
leftmost two panels in Figure \ref{F:inflation} for the fatgraph and its corresponding surface.
Each boundary component of $F(\mathbb{G})$ determines a closed edge-path  or cycle on $G$, oriented with the surface
lying on its left.  In particular,  a neighborhood of each edge inherits an orientation
from that of $F(\mathbb{G})$ which combine to give the oriented cycles as depicted in the third panel of
Figure \ref{F:inflation}.  Without affecting topological type of the constructed surface, one
may collapse each backbone to a single vertex with the induced fattening called the polygonal model of the
RNA, as illustrated in the rightmost
panels in Figure \ref{F:inflation}.  It is the orientation of each
backbone from the $5$'end to the $3'$ end that allows us to transform the
fatgraph of an RNA-structure or RNA-interaction into a fatgraph with one or
two vertices.

\begin{figure}[ht]
\centerline{\epsfig{file=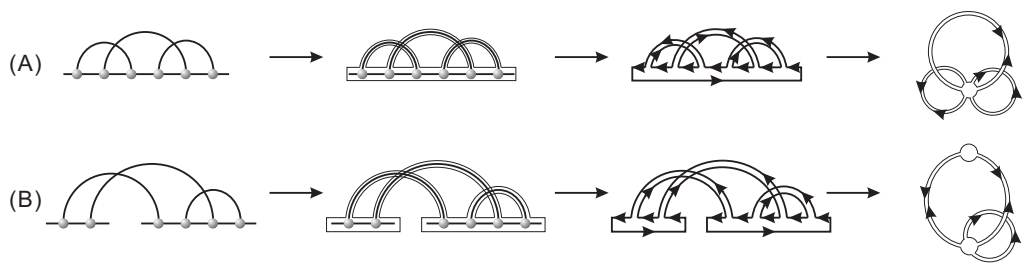,width=0.9\textwidth} \hskip8pt}
\caption{\small (A) The fatgraph of a diagram and its reduction to a single
vertex.  Contracting the backbone of a diagram into a single vertex decreases the length
of the boundary components and preserves the genus.
(B) Inflation of edges and vertices to ribbons and discs, as well as
walking along the boundary components. Here we have six vertices, seven edges
and one boundary component. The corresponding surface has Euler characteristic
$\chi=v-e=-1$ and $g=1$. At the last step, we collapse each backbone into a
single disc again preserving genus. The backbone of the polymer can be
recovered by inflating each disk to a backbone segment.
}\label{F:inflation}
\end{figure}

This backbone-collapse preserves
orientation, Euler characteristic and genus by construction. It is reversible by
inflating each vertex to form a backbone.
Using the collapsed fatgraph representation, we see that for a connected diagram
over $b$ backbones, the genus $g$ of the surface (with boundary) is determined
by the number $n$ of arcs as well as the number $r$ of boundary
components, namely, $2-2g-r = v-e = b-n$, cf.\
Figure~\ref{F:inflation}.

Diagrams over one and two backbones are related by gluing, i.e., we have the
mapping
$$
\alpha\colon \mathbb{E} \rightarrow \mathbb{D},
$$
where $\alpha(E)$ is obtained by keeping all arcs in $E$ and connecting the
$3'$ end of $R$ and the $5'$ end of $S$, see Figure~\ref{F:gluing} (A).

\begin{figure}[ht]
\centerline{\epsfig{file=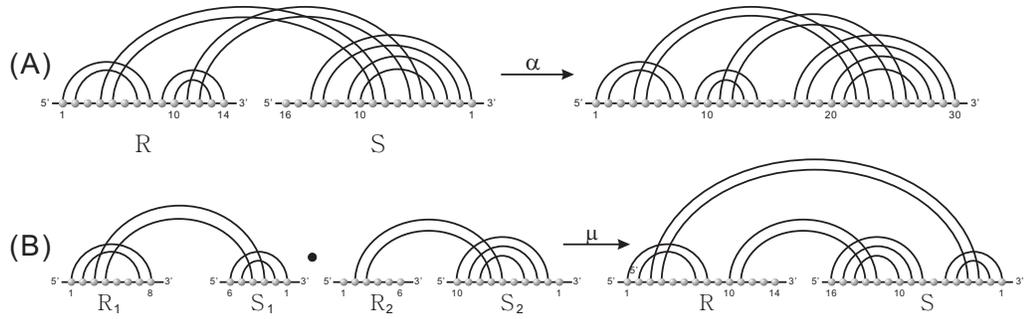,width=0.9\textwidth} \hskip8pt}
\caption{\small (A) Mapping a diagram over two backbones into a
diagram over one backbone by gluing. (B)
Mapping from two diagrams over two backbones to a diagram over two backbones by concatenating $R_2$
after $R_1$ and $S_1$ after $S_2$ preserving the orientation.
}\label{F:gluing}
\end{figure}
In addition to gluing, there is another operation mapping a pairs of diagram
over two backbones into a diagram over two backbones: given two diagrams over
two backbones, $E_1,E_2\in \mathbb{E}$ we can insert $E_2$ into the gap of
$E_1$ by concatenating the backbones $R_2$ and $R_1$ and $S_1$ and $S_2$
preserving orientation.; see Figure \ref{F:gluing} (B).
This composition is by construction again a diagram over two backbones denoted
$E_1\bullet E_2$, i.e., we have a mapping
\begin{equation}
\mu\colon \mathbb{E}\times \mathbb{E}\longrightarrow \mathbb{E},
\quad \mu(E_1,E_2)=E_1\bullet E_2.
\end{equation}
It is straightforward to see that $\bullet$ is an associative product with
unit given by the diagram over two empty backbones. The product $\bullet$ is
not commutative.

\section{Shadows}

\begin{definition}  A stack in a diagram is a maximal collection of parallel arcs
of the form $(i,j),(i+1,j-1),\ldots, (i+(\ell-1),j-(\ell-1))$.  An arc is non-crossing
if there is no other arc in the diagram that crosses it, and a vertex is isolated
if it has no arcs incident upon it.
A shadow is a diagram with no non-crossing arcs or isolated vertices so that each
stack has size one, and a shadow is non-trivial provided each backbone contains at
least one paired vertex.
\end{definition}

A diagram determines a shadow by removing all non-crossing arcs, deleting all isolated vertices and
collapsing each induced stack
to a single arc as in Figure~\ref{F:shadow}.
We shall denote the shadow of a diagram $X$ by $\sigma(X)$, so
$\sigma^2(X)=\sigma(X)$.  Projecting into the shadow does not affect
genus, i.e., $g(X)=g(\sigma(X))$.
In case there are no crossing arcs, $\sigma(X)$ becomes an empty
diagram on the same number of backbones as $X$ as in Figure~\ref{F:shadow} (C).
By definition, any empty backbone contributes one boundary
component. For example, for a diagram $X$ over $b$ backbones that
contains no crossing arcs, $\sigma(X)$ is a sequence of $b$ empty backbones with
$b$ boundary components.

\begin{figure}[ht]
\centerline{\epsfig{file=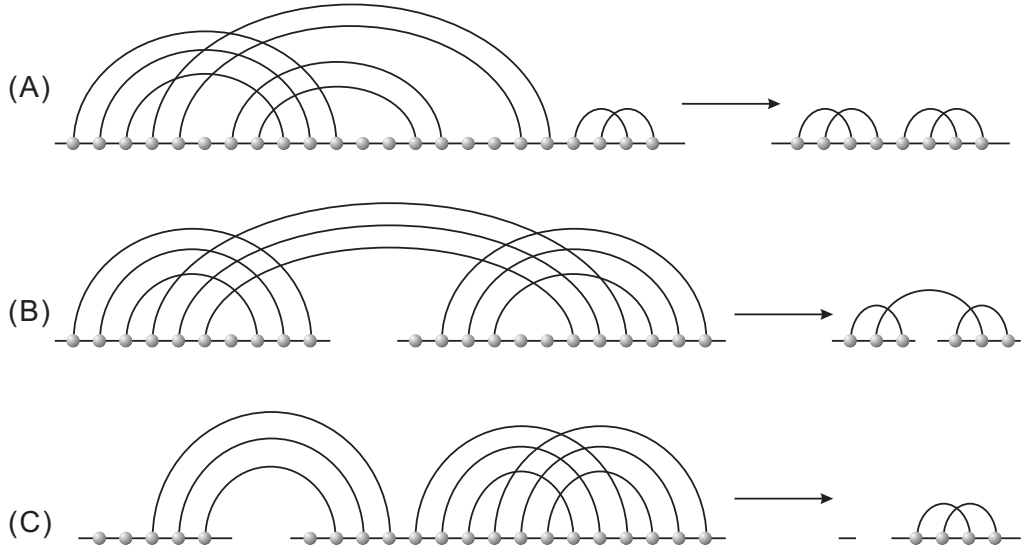,width=0.9\textwidth} \hskip8pt}
\caption{\small Shadows: (A) A diagram over one backbone and
its shadow (B) A diagram over two backbones whose shadow is again over two backbones
and (C) a shadow with an empty backbone.
}\label{F:shadow}
\end{figure}

Let us begin by refining an observation about shadows over one
backbone from \citet{Reidys:11a}:

\begin{theorem}\label{T:finiteshadows}
Shadows of genus $g\geq 1$ over one backbone have the following properties:\\
{\bf (a)}
a shadow of genus $g$ contains at least $2g$ and at most $(6g-2)$ arcs; in particular
for fixed genus, there are only finitely many shadows;\\
{\bf (b)} for any $2g\le \ell\le 6g-2$, there exists a shadow of genus $g$ containing
exactly $\ell$ arcs.
\end{theorem}
\begin{proof}
First note that if there is more than one boundary component, then there
must be an arc with different boundary components on its two sides and
removing this arc decreases $r$ by exactly one while preserving $g$ since
the number of arcs is given by $n=2g+r-1$.
Furthermore, if there are $\nu_\ell$ boundary components of length
$\ell$ in the polygonal model, then $2n=\sum_\ell \ell\nu_\ell$ since
each side of each arc is traversed once by the boundary.  For a shadow,
$\nu_1=0$ by definition,  and $\nu_2\leq 1$ as one sees directly.
It therefore follows that $2n=\sum_\ell \ell\nu_\ell\geq 3(r-1)+2$, so
$2n=4g+2r-2\geq 3r-1$, i.e., $4g-1\geq r$.
Thus, we have $n=2g+(4g-1)-1=6g-2$, i.e., any shadow can contain at most
$6g-2$ arcs. The lower bound $2g$ follows directly from $n=2g+r-1$ since
$r\geq 1$.

Let $S_{2g}$ be a shadow containing $2g$ mutually crossing arcs,
i.e., each arc crosses any of the remaining $(2g-1)$ arcs. $S_{2g}$
has genus $g$ and contains a unique boundary component of
length $4g$, i.e., traversing $4g$ non-backbone arcs counted with multiplicity.
We construct a new shadow $S_{2g+1}$ of genus $g$ containing $2g+1$ arcs,
by inserting an arc crossing into $S_{2g}$ from the $5'$ end of
$S_{2g}$ such that the boundary component in $S_{2g}$ splits into
one boundary component of length $3$ and another of length $4g+2-3=4g-1$.
The latter becomes the first boundary component of $S_{2g+1}$.
The newly inserted arc is by construction crossing, splits a boundary
component and preserves genus.
We now prove the assertion by induction of the number of inserted arcs.
By the induction hypothesis, there exists a shadow $S_{2g+i}$ of genus $g$ having $2g+i$ arcs,
whose first boundary component has length $4g-i$.
Again, we insert a
crossing arc as described above thereby splitting the first boundary component
into one of length $3$ and the other of length $(4g-(i+1))$. After $i=4g-2$
such insertions, we arrive at a shadow whose first boundary component has
length $2$ while all other boundary components have length $3$.
Accordingly, there exists a set
$\{S_{2g},S_{2g+1},\ldots,S_{2g+(4g-2)}\}$ of shadows all having genus $g$, where
each $S_j$ contains $j$ arcs, see Figure~\ref{F:newbond}.
\end{proof}

\begin{figure}[ht]
\centerline{\epsfig{file=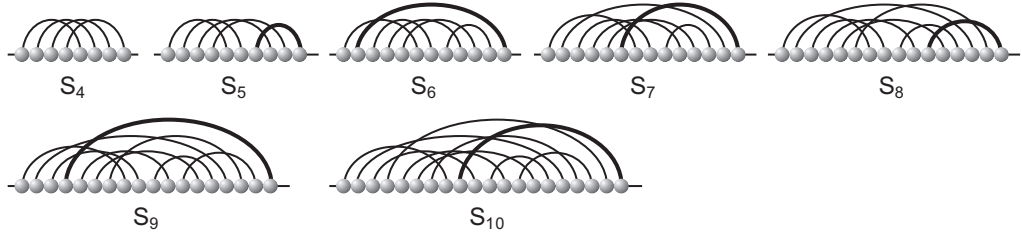,width=0.9\textwidth} \hskip8pt}
\caption{\small
Constructing the sequence of shadows $S_{\ell}$ for genus $g=2$,
see Theorem~\ref{T:finiteshadows}, for $2g=4\le \ell\le 6g-2=10$.
Newly inserted arcs are drawn bold.
}\label{F:newbond}
\end{figure}

\begin{corollary}\label{C:22}
A shadow over two backbones has the following properties:\\
{\bf (a)}
 a shadow of genus $g\ge 1$ over two backbones contains at least $(2g+1)$
and at most $6(g+1)-2$ arcs; a shadow of genus $0$ has at least $2$ and
at most $4$ arcs.
in particular, the set of such shadows is finite;\\
{\bf (b)}
for any $(2g+1)\le \ell\le 6(g+1)-2$ in case of $g\ge 1$ and
$2\le \ell \le 4$ in case of $g=0$, there exists some shadow over two
backbones with genus $g$ containing exactly $\ell$ arcs.
\end{corollary}

\begin{proof}
We first claim that any shadow of genus $g$ over two backbones can be obtained by cutting
the backbone of a shadow over one backbone having either genus $g$ or ${g+1}$.
To see this, suppose we are given a shadow of genus $g$, having $r$
boundary components and $n$ arcs so that $2-2g-r=b-n$, i.e., $g=(2+n-r-b)/2$,
where $b=1$.
Cutting the backbone then either splits a boundary component, or merges two
distinct boundary components.
Since cutting does not affect arcs and increases the number of backbones by one
we have the resulting genus
$$
g'=(2+n-(r+1)-(b+1))/2=g-1 \quad \text{or}\quad  g'=(2+n-(r-1)-(b+1))/2=g
$$
as was claimed.
We next observe that a shadow of genus $g=0$ over two backbones has at least $2$
arcs, while the maximum number of arcs contained in such a shadow is given by
$6(0+1)-2=4$.
For $g\ge 1$, it is impossible to cut a shadow of genus $g$ having
$2g$ arcs and keep the genus. Thus the shadow of genus $g$ over two backbones
has at least $2g+1$ arcs.
We can always map an arbitrary shadow over two backbones of genus $g$ via
$\alpha$ into a shadow over one backbone, whence the assertion.
Theorem~\ref{T:finiteshadows} guarantees that there are only finitely many
such shadows, and the corollary follows.
\end{proof}

\begin{corollary}\label{C:seven}
There exist exactly seven non-trivial shadows over two backbones having genus $0$.
\end{corollary}
\begin{proof}
There exists no non-trivial shadow over one backbone of genus $0$ since $0$-structures
over one backbone are secondary structures containing exclusively
non-crossing arcs.
In view of Corollary~\ref{C:22}, all non-trivial shadows over two backbones having genus
$0$ are therefore obtained by cutting the backbone of shadows of genus $1$
over one backbone. By inspection, there are seven possible such cuts as in
Figure~\ref{F:cut_backbone}.
\end{proof}

\begin{figure}[ht]
\centerline{\epsfig{file=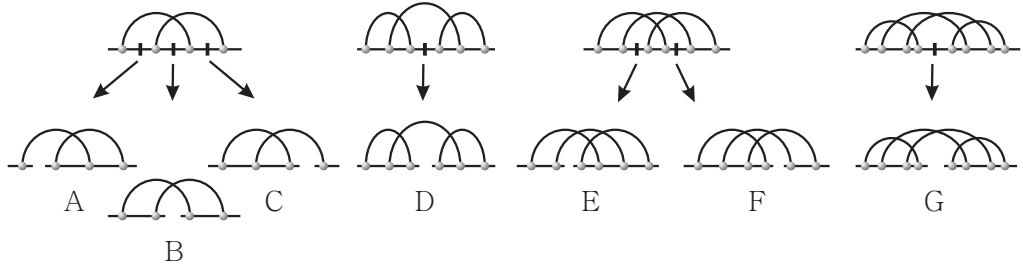,width=0.9\textwidth} \hskip8pt}
\caption{\small The shadows over two backbones having genus $0$ obtained
by cutting the four shadows of genus $1$ over one backbone.
}\label{F:cut_backbone}
\end{figure}

\section{Irreducibility}

\begin{definition}
A diagram $E$ over $b$ backbones is called irreducible if and only if it
is connected and for any two arcs, $\alpha_1,\alpha_k$ contained in $E$,
there exists a sequence of arcs $(\alpha_1,\alpha_2,\dots,\alpha_{k-1},\alpha_k)$
such that $(\alpha_i,\alpha_{i+1})$ are crossing.
\end{definition}

We proceed by refining Theorem~\ref{T:finiteshadows}:
\begin{corollary}\label{P:finiteshadows}
An irreducible shadow having genus $g=0$ over two backbones contains at least
$2$ and at most $4$ arcs, and for and $2\le \ell \le 4$,
there exists an irreducible shadow of genus $g=0$
over two backbones having exactly $\ell$ arcs.
An irreducible shadow having genus $g\ge 1$ has the following properties:\\
{\bf (a)}
every irreducible shadow with genus $g$ over two backbones contains at least
$2g+1$ and at most $6(g+1)-2$ arcs;\\
{\bf (b)}
for arbitrary genus $g$ and any $2g+1\le\ell\le 6g-2$, there exists an
\emph{irreducible} shadow of genus $g$ over one backbone having exactly $\ell$
arcs.
\end{corollary}
\begin{proof}
Part a) follows directly from Theorem~\ref{T:finiteshadows}, and for b), the shadows
$S_{2g+1},\ldots, S_{6g-2}$ generated in the proof of
Theorem~\ref{T:finiteshadows}, are in fact
{irreducible} as in  Figure~\ref{F:newbond}.
\end{proof}

\begin{definition}Let $X$ be a diagram.  We call $S'$ an irreducible shadow of $X$ (irreducible
$X$-shadow) if and only if $S'$ is an irreducible shadow and any arc in $S'$ is
contained in $X$. Let $\mathbb{I}(X)=\{S'\subset X\mid S'\; \text{\rm is an
irreducible $X$-shadow}\;\}$.
\end{definition}

Clearly, our notion of irreducibility recovers for diagrams over one
backbone that of \citet{Reidys:11a,Bon:08}.
A diagram $D$ over one backbone can iteratively be
decomposed by first removing all non-crossing arcs as well as isolated vertices
and second by removing irreducible $D$-shadows iteratively as follows: \\
$\bullet$ one removes (i.e., cuts the backbone at two points and after removal
          merges the cut-points) irreducible $D$-shadows from bottom to top,
          i.e., such that there exists no irreducible $S$-shadow that is nested
          within the one previously removed. \\
$\bullet$ if the removal of an irreducible $D$-shadow induces the formation
          of a non-trivial stack as in Figure~\ref{F:shadow1}, then it is collapsed into a
          single arc.\\

\begin{figure}[ht]
\centerline{\epsfig{file=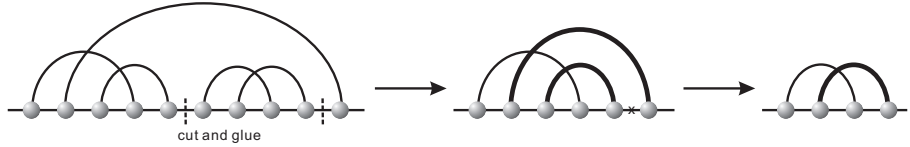,width=0.8\textwidth} \hskip8pt}
\caption{\small Removing irreducible shadows from ``bottom to top''. Any
stacks, that are induced by these removals are collapsed into single arcs.
}\label{F:shadow1}
\end{figure}


We next extend the decomposition of diagrams over one backbone \citet{Reidys:11a}
to diagrams over two backbones. Let $E$ be a diagram over two backbones.
By definition, irreducible $E$-shadows over two backbones are either
connected or a disjoint union of two irreducible shadows over one backbone.
Thus, $E$ can be decomposed by removing first all non-crossing arcs as
well as any isolated vertices and second all irreducible $E$-shadows in
two rounds as follows: \\
$\bullet$ remove any irreducible $E$-shadows over one backbone, from bottom
          to top, as previously described, see Figure~\ref{F:irre2}, \\
$\bullet$ remove the irreducible $E$-shadows over two backbones iteratively,
          starting with the irreducible $E$-shadow containing the leftmost
          vertex of the second backbone, see Figure~\ref{F:irre2}.

\begin{figure}[ht]
\centerline{\epsfig{file=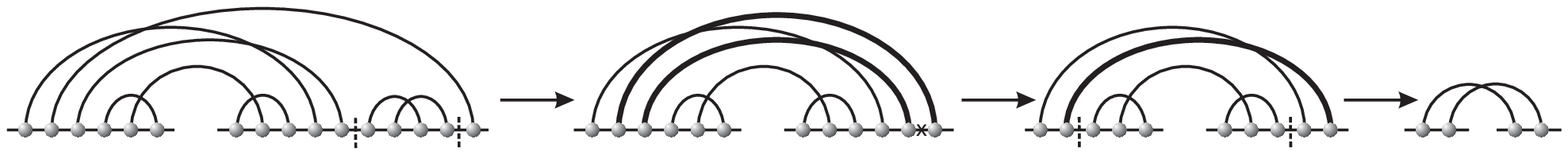,width=0.9\textwidth} \hskip8pt}
\caption{\small Decomposition of a shadow over two backbones. First, from
bottom to top, the only irreducible shadow over one backbone is removed.
During its removal, a stack of length two is induced (bold arcs),
which is projected into a single arc.
Second, the two irreducible shadows over two backbones are iteratively
removed.}
\label{F:irre2}
\end{figure}

\section{$\gamma$-structures over two backbones}

\begin{definition}
A diagram $X$ over $b$ backbones is a $\gamma$-structure over $b$ backbones if and
only if we have $g(S')\le \gamma$ for any irreducible $X$-shadow $S'$.
\end{definition}
With foresight, we  refine the notion of irreducible $X$-shadow as follows:
\begin{eqnarray*}
\mathbb{I}_1(E) & = &
 \{ \, S'\mid \text{$S'$ is an irreducible $E$-shadow over one backbone}\,\} ,\\
\mathbb{I}_2^{i}(E) & = &
    \{ \, S'\mid \text{$S'$ is an irreducible $E$-shadow over two backbones,
               where $g(\alpha(S'))=g(S')+i$} \, \}.
\end{eqnarray*}
\begin{lemma}\label{L:genus_cal}
Suppose $E$ is a $\gamma$-structure over two backbones. Then
\begin{equation}
g(E) =
\begin{cases}
\sum_{S' \in \mathbb{I}_1(E)} g(S')+
\sum_{S'\in \mathbb{I}_2^0(E)} g(S')+\sum_{S'\in \mathbb{I}_2^1(E)} (g(S')+1),
& \text{if} \quad \mathbb{I}_2^0(E)\ne \emptyset ;\\
\displaystyle{\sum_{S' \in \mathbb{I}_1(E)} g(S')+
                                     \sum_{S'\in \mathbb{I}_2^1(E)} (g(S')+1) -1},
&  \text{if} \quad \mathbb{I}_2^0(E) = \emptyset.
\end{cases}
\end{equation}
\end{lemma}
\begin{proof}
By construction, $\alpha(E)$ is a shadow over one backbone consisting of
irreducible components of genus at most $\gamma+1$. Thus, $\alpha(E)$
is a $(\gamma+1)$-structure and
\begin{eqnarray}\label{E:erni}
g(\alpha(E)) & = & \sum_{S'\in \mathbb{I}_1(E)}g(S')+\sum_{S'\in \mathbb{I}_2^0(E)} g(S')
+\sum_{S' \in \mathbb{I}_2^1(E)} (g(S')+1).
\end{eqnarray}
Let ${\mathbb S}_1=\mathbb{S}_1(E)$ be the set of $E$-subshadows over two backbones where the
backbones are on the same boundary component and
let ${\mathbb S}_2=\mathbb{S}_2(E)$ be those that are not. We have
\begin{equation}
g(S')=
\begin{cases}
g(\alpha(S')), & \text{iff}   \quad S'\in \mathbb{S}_1(E); \\
g(\alpha(S'))-1, & \text{iff} \quad S'\in \mathbb{S}_2(E).
\end{cases}
\end{equation}
{\it Claim $1$.} Suppose $\mathbb{I}_2^0(E)=\varnothing$, then
\begin{equation}
g(E)= \sum_{S' \in \mathbb{I}_1(E)} g(S')+\sum_{S'\in \mathbb{I}_2^1(E)} (g(S')+1) -1.
\end{equation}
To prove this, we use the operation $S_1\bullet S_2\in \mathbb{S}_2$.
By associativity of $\bullet$, we conclude
that $E$ has both backbones on the same boundary component, i.e.,
\begin{equation}
g(E)=g(\alpha(E))-1,
\end{equation}
and in view of eq.~(\ref{E:erni}), Claim $1$ follows.\\
{\it Claim $2$.} If $\mathbb{I}_2^0(E)\neq \varnothing$, then
\begin{equation}
g(E)= \sum_{S' \in \mathbb{I}_1(E)} g(S')+\sum_{S'\in \mathbb{I}_2^1(E)} (g(S')+1).
\end{equation}
We claim that $\mathbb{I}_2^0(E)\neq \varnothing$ implies $g(E)=g(\alpha(E))$.
Indeed, $\mathbb{I}_2^0(E)\neq \varnothing$ guarantees that there exists some
irreducible shadow $S_0' \in \mathbb{I}_2^0(E)$. $S_0'$ has by definition the
property $g(\alpha(S_0'))=g(S_0')$, i.e., gluing the two $S_0'$-backbones
does not merge boundary components, whence $S_0' \in \mathbb{S}_1$.
Now, at some point $S_0'$ appears as a factor in the shadow of $E$ which
implies $E\in \mathbb{S}_1$. Accordingly, we have  $g(E)=g(\alpha(E))$,
from which it follows that
\begin{equation}
g(E)= \sum_{S' \in \mathbb{I}_1(E)} g(S')+\sum_{S'\in \mathbb{I}_2^1(E)} (g(S')+1).
\end{equation}
\end{proof}


\section{A grammar for $0$-structures over two backbones}


In this section, we develop an unambiguous decomposition grammar
$\mathscr{G}_0$ for $0$-structures over two backbones or $0_2$-structures.
$0_2$-structures map via $\alpha$ into $1$-structures over one backbone
of genus zero or one.
In order to formulate $\mathscr{G}_0$, let us recall that we draw the oriented
backbones $R$ and $S$ horizontally and consecutively starting with the $5'$
end of $R$ or $R_1$ and ending with the $3'$ end of $S$ or $S_1$.
We denote a structure over two backbones by $J^{I}_{i,j;h,\ell}$, where $i$, $j$
are vertices contained in $R$ and $h$, $\ell$ are contained in $S$.
In particular, we shall write $[i,i]$ for a single vertex letting $[i,i-1]$
represent an ``empty'' backbone.
For instance, $J^I_{i,i-1;h,\ell}$ denotes the structure over one backbone on
the interval $[h,\ell]$ on $S$, where $h\le \ell$, $J^I_{i,j;h,h-1}$ denotes the
structure over one backbone on the interval $[i,j]$ on $R$, where $i \le j$,
and $J^I_{i,i-1,h,h-1}=\varnothing$.

The key building blocks of $\mathscr{G}_0$ are the following:
\begin{itemize}
\item \emph{gap-structures:}
      a gap structure $J^{G}_{i,j;h,\ell}$ is a secondary structure over
      $[i,\ell]$ with a gap from $j$ to $h$ such that $(i,\ell)$ and
      $(j,h)$ are base pairs; within the two gaps, there are no crossing arcs.
\item \emph{hybrid-structures:} a hybrid structure
      $J^{Hy}_{i_1,i_\ell;j_1,j_\ell}$ is a maximal
sequence of intermolecular interior loops consisting of exterior arcs
$R_{i_1}S_{j_1},\dots, R_{i_\ell}S_{j_\ell}$ where $R_{i_h}S_{j_h}$ is nested within
$R_{i_{h+1}}S_{j_{h+1}}$ and where the internal segments $R[i_h+1,i_{h+1}-1]$ and
$S[j_h+1,j_{h+1}-1]$ consist of single-stranded nucleotides only;
that is, a hybrid structure (hybrid) is the maximal unbranched stem-loop formed
by external arcs.
\item \emph{tight structures:}
a tight structure (TS) $J^T_{i,j;h,\ell}$ is a structure in which the four
positions, $i$, $j$, $h$ and $\ell$ are endpoints of an irreducible shadow over two backbones.
\item \emph{pre-tight structures:}
a pre-tight structure (PTS) is a structure $J^{PT}_{i,j;h,\ell}$,
containing a tight structure $J_{i_1,j;h_1,\ell}$ or a hybrid structure
$J^{{Hy}_{i_1,j;h_1,\ell}}$ for some $i_1\ge i$ and $h_1\ge h$.
\end{itemize}
\begin{figure}[ht]
\centerline{\epsfig{file=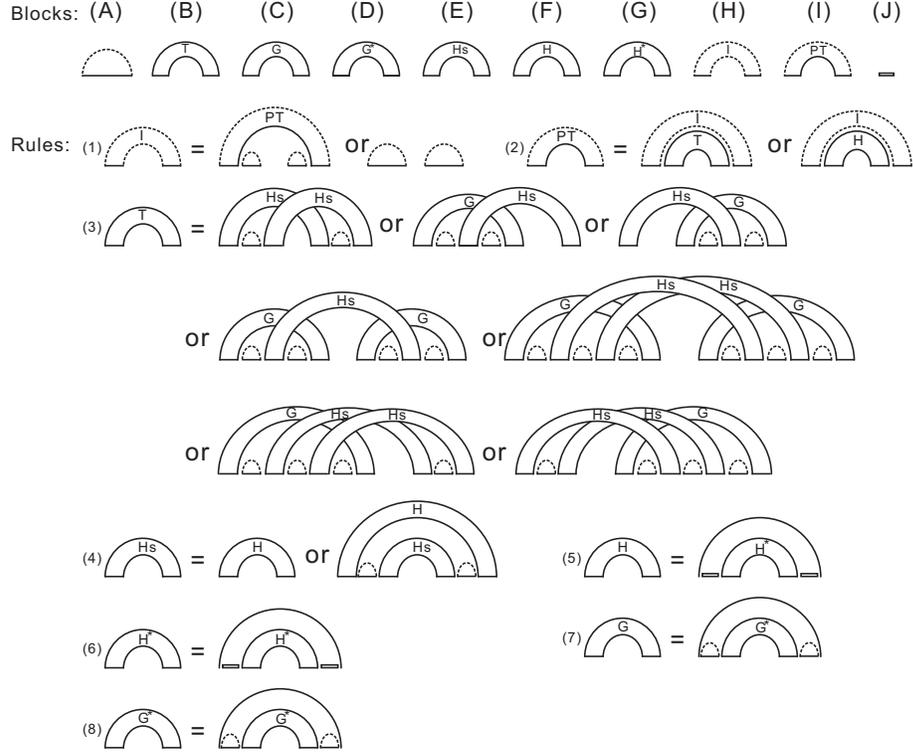,width=0.8\textwidth} \hskip8pt}
\caption{\small The grammar $\mathscr{G}_0$:
{\bf (A):} a secondary structure over $[i,j]$,
{\bf (B):} a tight structure $J^T_{i,j;r,s}$,
{\bf (C):} a gap structure $J^G_{i,j;r,s}$ over one backbone,
{\bf (D):} a substructure of a gap structure $J^{G^*}_{i,j;r,s}$ such that $(i,s)$ and
$(j,r)$ are interior arcs but itself is not a maximal gap structure,
{\bf (E):} a substructure $J^{Hs}_{i,j;r,s}$ consist of hybrid structures and secondary
structures,
{\bf (F):} a hybrid structure $J^{Hy}_{i,j;r,s}$,
{\bf (G)} a substructure $J^{Hy^*}_{i,j;r,s}$
of hybrid structure such that $(i,j)$ and $(r,s)$ are exterior arcs but itself
is not a hybrid structure because it is not maximum,
{\bf (H):} an arbitrary structure on two backbones,
{\bf (I):} a pre-tight structure $J^{PT}_{i,j;r,s}$,
{\bf (J):} an open structure consisting of unpaired bases, 
{\bf (1)--(8)}: decomposition rules for the previously defined blocks.
}\label{F:grammar2}
\end{figure}
Now we are in position to formulate the production rules of $\mathscr{G}_0$,
detailed in Figure~\ref{F:grammar2}:\\

{\bf (1)}: given an arbitrary structure $J^I_{i,j;h,\ell}$, we remove starting
from $j$ and $\ell$ secondary structure blocks until an exterior arc is
encountered; such an exterior arc is contained in a pre-tight structure and
otherwise, $J^I_{i,j;h,\ell}$ contains no exterior arc and thus decomposes into
two disjoint secondary structures; \\
{\bf (2)}: the decomposition of pre-tight structures $J^{PT}_{i,j;h,\ell}$: if
$R_jS_{\ell}$ is an exterior arc, then it is decomposed into a hybrid
$J^{{Hy}}_{i_1,j;h_1,\ell}$ and an arbitrary substructure $J^I_{i,i_1-1;h,h_1-1}$;
otherwise, it is decomposed into a tight structure $J^T_{i_1,j;h_1,\ell}$ and an
arbitrary structure $J^I_{i,i_1-1;h,h_1-1}$; \\
{\bf (3)}: in case of tight structures depending on which type of shadow is
contained in the tight structure, there are $7$ ways to disect into maximal
gap structures and hybrid-structures (which in turn collapses into interior
and exterior arcs of the irreducible shadow, respectively),
as well as secondary structures;\\
{\bf (4)}: a substructure $J^{Hs}_{i,j;h,\ell}$ consists of hybrids and secondary
structures, where each hybrid structure is maximal.;\\
{\bf (5)}: a maximal hybrid structure $J^{{Hy}}_{i,j;h,\ell}$ is decomposed into an
exterior arc $R_iS_h$ and a non-maximal hybrid structure
$J^{{Hy}^*}_{i_1,j;h_1,\ell}$ with $i<i_1<j$ and $h<h_1<\ell$; \\
{\bf (6)}: a non-maximal hybrid structure $J^{{Hy}^*}_{i,j;h,\ell}$ is decomposed
into an exterior arc $R_iS_h$ and a non-maximal hybrid structure
$J^{{Hy}^*}_{i_1,j;h_1,\ell}$ with $i<i_1<j$ and $h<h_1<\ell$.;\\
{\bf (7)}: a maximal gap structure $J^G_{i,j;h,\ell}$ is decomposed via the
context-free grammar for secondary structures assuming that there is a
virtual hairpin loop in $[j,h]$; note that the substructure decomposed by a
maximal gap structure is no longer maximal; we use $J^{G^*}_{i,j;h,\ell}$ to
denote such a non-maximal gap structure derived via this decomposition; \\
{\bf (8)}: a non-maximal gap structure $J^{G^*}_{i,j;h,\ell}$ is decomposed similarly to
the decomposition of a maximal gap structure.

\begin{lemma}\label{L:cover}
Any $0$-structure over two backbones can uniquely be decomposed via
$\mathscr{G}_0$, and any diagram generated by $\mathscr{G}_0$ is a
$0$-structure over two backbones.
\end{lemma}
\begin{proof}
First, we show that a $0_2$-structure can uniquely be decomposed into blocks
containing exclusively non-crossing arcs. We shall establish this by induction
on the number of its irreducible shadows. \\
\emph{Induction basis:} any $0_2$-structure over two backbones that contains no
shadow of genus zero over two backbones exhibits no crossing arcs. Namely, it
contains only blocks that are either secondary structures or hybrids.
Accordingly, such a structure can be decomposed uniquely via the context-free
grammar of secondary structures or the unique decomposition of
hybrid-structures.\\
\emph{Induction step:}
Suppose $E_m$ is a $0_2$-structure containing $m\geq 1$ irreducible shadows
over two backbones of genus $0$.
We decompose from ``inside to outside'', i.e., from the $3'$-end of $R$ and
the $5'$-end of $S$. Suppose we encounter a substructure $S$ which collapses
into an irreducible shadow over two backbones of genus $0$.
$S$ itself determines a unique maximal tight structure, $T_S$, such that
$\sigma(T_S)=S$. Removing $T_S$ from $E_m$ yields a diagram $E_{m-1}$ over two backbones
containing $m-1$ irreducible shadows over two backbones of genus
$0$. The induction hypothesis guarantees the unique decomposition of
$E_{m-1}$ via $\mathscr{G}_0$. \\
It remains to show how to decompose tight structures: the shadow
of a tight structure is by construction irreducible and is given by one of the seven
irreducible shadows over two backbones described in Corollary~\ref{C:seven}.
In order to decompose a tight structure, we dissect it into maximal gap
structures and hybrid-structures (which in turn collapse into interior and
exterior arcs of the irreducible shadow, respectively), as well as secondary
structures.
All of these elements are $\mathscr{G}_0$-blocks that do not contain any
crossing arcs and can therefore be decomposed via a modified version of the
context-free grammar of secondary structures, described above.
Accordingly, there are seven ways to uniquely decompose a tight structure
into blocks containing exclusively non-crossing arcs.\\
Finally,  we show that $\mathscr{G}_0$ generates only $0_2$-structures.
By construction, $\mathscr{G}_0$ constructs tight structures via secondary
structure blocks, gap-structures and hybrid-structures. It furthermore
generates via the insertion of secondary structure blocks, hybrid structures
and tight structures. Thus, any structure generated by $\mathscr{G}_0$ is a
$0_2$-structure, whence the lemma.
\end{proof}

\begin{theorem}
The grammar $\mathscr{G}_0$ has the following properties:\\
{\bf (a)} $\mathscr{G}_0$ is unambiguous;\\
{\bf (b)} $\mathscr{G}_0$ allows computation of the partition function, base pairing
          probabilities, the probability of hybrid-blocks, gap-structures and
          Boltzmann sampling of $0_2$-structures,\\
{\bf (c)} $\mathscr{G}_0$ has a time $O(n^6)$ and space $O(n^4)$ complexity
          for generating the partition function of $0_2$-structures.
\end{theorem}
\begin{proof}
Assertion {\bf (a)} follows from Lemma~\ref{L:cover}. Consequently, $\mathscr{G}_0$
can be employed to count $0_2$-interaction structures over two backbones for given
sequences $R$ and $S$ as well as to compute the partition function
\begin{equation*}
Q=\sum_{s\in \mathfrak{J_{R,S}}} e^{-G(s)/RT}
\end{equation*}
of $0_2$-structures, where $R$ is the universal gas constant, $T$ is the temperature, $G(s)$ is
energy of structure $s$ over sequence $x$, and $\mathfrak{J_{R,S}}$ is the set
of $0$-interaction structures in which all base pairs $(i,j)$ satisfy the base pairing
rules for RNA, i.e., $(i,j)\in \{AU,UA,GC,CG,GU,UG\}$.

As for assertion {\bf (b)}, let $N_{i,j;h,\ell}$ denote the substructure
represented by the nonterminal symbol $N$ in $\mathscr{G}_0$ over $[i,j]$
and $[h,l]$, where $N=\{I,PT,T,Hs,{Hy},{Hy}^*,G,G^*\}$.
Note that secondary structures are presented by an arbitrary structure $I$
setting one backbone empty.
For each of these symbols, we introduce corresponding partial partition functions
$Q_{N_{i,j;h\ell}}$. Since $\mathscr{G}_0$ is unambiguous, the recursions for the
partial partition
functions are derived by replacing minima by sums and addition of energy
contribution by multiplication of partial partition functions, see
e.g., \citet{Voss:06}. For instance, the recursion for the partition functions
corresponding to the nonterminal symbol $PT$ reads
\begin{equation*}
 Q_{J^{PT}_{i,j;h.\ell}} =\sum_{k_1,k_2} Q_{J^I_{i,k_1;h,k2,}}\times Q_{J^T_{k_1+1,j;k_2+1;\ell}}
 +\sum_{k_1,k_2} Q_{J^I_{i,k_1;h,k2,}}\times Q_{J^{{Hy}}_{k_1+1,j;k_2+1;\ell}}.
\end{equation*}
The probabilities $\mathbb{P}_{N_{i,j;h,\ell}}$ of partial substructures of type $N$
are readily calculated from the
partial partition functions. These ``backward recursions'' are analogous to
those derived by \citet{McCaskill} for secondary structures without crossings.  It follows that we have
\begin{equation*}
\mathbb{P}_{N_{i,j}}=\sum \mathbb{P}_s,
\end{equation*}
where the sum is over
all $0_2$-interaction structures containing
$N_{i,j;h,\ell}$.

Suppose $N_{i,j;h,\ell}$ is obtained by decomposing
$\theta_s$. The conditional probabilities $\mathbb{P}_{N_{i,j;h,\ell}|\theta_s}$ are then
given by $Q_{\theta_s}(N_{i,j;h,\ell})/Q_{\theta_s}$, where $Q_{\theta_s}$ represents the
partition function of $\theta_s$ and $Q_{\theta_s}(N_{i,j;h,\ell})$ represents the
partition functions for those $\theta_s$-configurations that contain $N_{i,j;h,\ell}$.
Taking the sum over all possible $\theta_s$, we obtain
\begin{equation*}
 \mathbb{P}_{N_{i,j;h,\ell}}=\mathbb{P}_{\theta_s}
\frac{Q_{\theta_s}(N_{i,j;h,\ell})}{Q_{\theta_s}}.
\end{equation*}
From this backward recursion, one immediately derives a stochastic backtracing
recursion from the probabilities of partial structures that generates a
Boltzmann sample of $0$-interaction structures; see
\citet{Tacker:96a,Ding:03,Huang:10a} for similar constructions.
The basic data structure for this sampling is a stack $A$ which stores blocks
of the form $(i,j;r,s,N)$, presenting interaction substructures of nonterminal
symbols $N$. $L$ is a set of base pairs storing those removed by the
decomposition step in the grammar. We initialize with the block $(1,n,I)$
in $A$, and $L=\varnothing$. In each step, we pick up one element in $A$
and decompose it via the grammar with probability $Q^M/Q^N$, where $Q^N$
is the partition function of the block which is picked up from $A$, and
$Q^M$ is the partition function of the target block which is decomposed by
the rewriting rule. The base pairs which are removed in the decomposition step
are moved to $L$. For instance for the decomposition rule of $J^{PT}_{i,j;h,\ell}$,
decomposing block $(i,j,PT)$ into the two blocks:
$(i,k_1;h,k_2,I)$ and $(k_1+1,j;k_2+1,\ell,T)$,
for fixed indices $k_1$, $k_2$, the probability of decomposing $(i,j,PT)$ reads
\begin{equation*}
\mathbb{P}_{k_1,k_2}=\frac{Q_{J^I_{i,k_1;h,k_2}}\times Q_{J^{T}_{k_1+1,j;k_2+1,\ell}}}
{Q_{J^{PT}_{i,j;h,\ell}}}.
\end{equation*}
The sampling step is iterated until $A$ is empty. The resulting
$0_2$-interaction structure is given by the list $L$ of base pairs.
The probability of hybrid-structures can be calculated since a hybrid structure
is by construction a block in the grammar, see \citet{Huang:10a}.
The probability of interactions involving a fixed interval $[i,j]$ is given by
$$
\mathbb{P}^{\mathsf{target}}_{[i,j]}=\sum_{h,\ell}\mathbb{P}^{{Hy}}_{i,j;h.\ell}.
$$
A gap structure, representing a maximal non-crossing stem on either backbone
is also a $\mathscr{G}_0$-block, whence its probability is readily computable.
Similarly, the probability of parings within the same backbone for a fixed
interval $[i,j]$ can be expressed as:
$$
\mathbb{P}^{\mathsf{paring}}_{[i,j]}=\sum_{h,\ell}\mathbb{P}^{G}_{i,j;h.\ell}.
$$
In order to prove assertion {\bf (c)}, we observe that any product of two blocks
has $O(n^6)$ time complexity. We conclude from this that all $\mathscr{G}_0$-rules,
except for {\bf (3)} and {\bf (4)} are of $O(n^6)$ time complexity. It thus remains
to analyze {\bf (3)} and {\bf (4)}\footnote{which are in fact  $O(n^{16})$ for
{\bf (3)} and $O(n^8)$ for {\bf (4)} time complexity as it stands}.
To this end, we introduce intermediate blocks whose function is
transitional storage.
\begin{itemize}
\item[1.] $J^U_{i,j;h,\ell}$ stores the result of the product $J^{{Hy}}_{i,i_1,h,h_1}$
and two secondary structure over interval $[i_1+1,j]$ and $[h_1+1,\ell]$ with
$i\le i_1\le j$ and $h\le h_1\le \ell$.
\item[2.] $J^V_{i,j;h,\ell}$ stores the result of the product $J^G_{i,i_1;h_1,\ell}$
and two secondary structure over interval $[i_1+1,j]$ and $[h+1,h_1]$ with
$i< i_1\le j$ and $h\le h_1< \ell$.
\item[3.] $J^W_{i,j;h,\ell}$ stores the result of the product $J^V_{i,i_1;j_1,j}$
and $J^{{Hy}}_{i_1+1,j_1-1;h,\ell}$ with
$i<i_1<j_1<j$.
\item[4.] $J^X_{i,j;h\ell}$ stores the result of the product $J^U_{i,i_1;h_1,\ell}$
and $J^{{Hy}}_{i_1+1,j;h,h_1-1}$ with $i<i_1<j$ and $h<h_1<\ell$.
\item[5.] $J^Y_{i,j;h\ell}$ stores the result of the product $J^V_{i,i_1;j_1,j}$
and $J^X_{i_1+1,j_1-1;h,\ell}$ with
$i<i_1<j_1<j$.
\end{itemize}
By virtue of these new blocks, we may rewrite {\bf (3)} and {\bf (4)} in terms of
{\bf (3')} and {\bf (4')} as displayed in
Figure~\ref{F:intermediate}.
\begin{figure}[ht]
\centerline{\epsfig{file=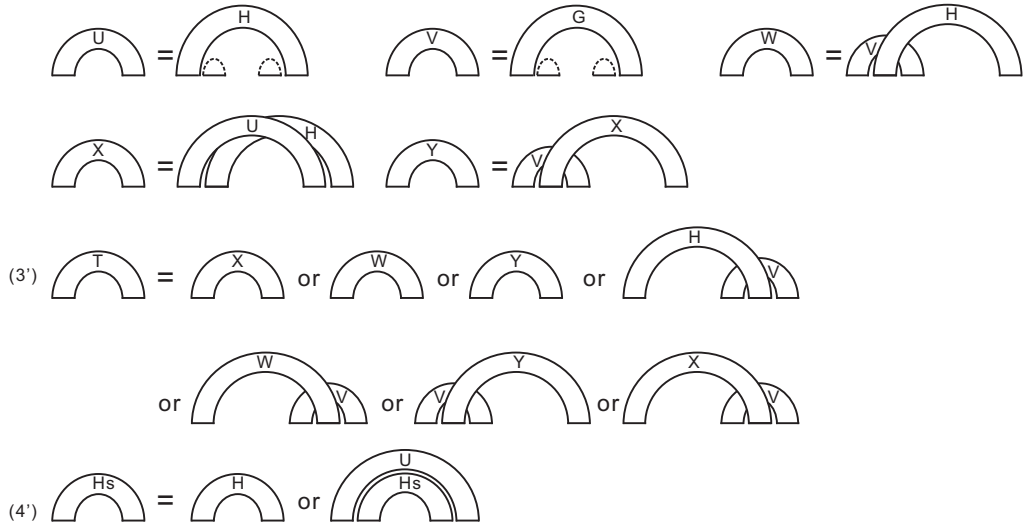,width=0.9\textwidth} \hskip8pt}
\caption{\small
The decomposition of $J^T_{i,j;h,\ell}$ and $J^{Hs}_{i,j;h,\ell}$ via the five
intermediate blocks $J^{U}_{i,j;h,\ell}$, $J^{V}_{i,j;h,\ell}$, $J^{W}_{i,j;h,\ell}$,
$J^{X}_{i,j;h,\ell}$ and $J^{Y}_{i,j;h,\ell}$.
They allow the decomposition of $J^T_{i,j;h,\ell}$ and $J^{Hs}_{i,j;h,\ell}$
with $O(n^6)$ time complexity.
}\label{F:intermediate}
\end{figure}
After including these five intermediate blocks, we obtain two additional,
nonterminal symbols in each decomposition rule. Since it requires two free
variables to have the product of two nonterminal symbols and at most four
variables to describe the two blocks, the decompositions in this form are of
$O(n^6)$ time complexity. We use at most $4$-dimensional matrices to store
the blocks in $\mathscr{G}_0$, whence the $O(n^4)$ space complexity.
\end{proof}

\section{Discussion}

In this paper,  we have introduced the toplogical filtration of RNA
interaction structures and developed the notions of shadows,
irreducibility and $\gamma$-structures for them. Shadows are of particular importance for the
minimum free energy folding since they represent the basic
motifs of genus $g$. Since we have proved that for any genus
there are always finitely many such shadows, it is therefore
in principle possible to assign them individual energies, which
would presumably lead to high specificity.

The simplest topological class of RNA interaction structures is
that of $0$-structures over two backbones.
This is the two-backbone analogue of the classical RNA secondary
structures. Despite their simple irreducible shadows
(Corollary~\ref{C:seven}), $0$-structures over two backbones
exhibit features not present in the AP-structures of
\citet{Pervouchine:04,Alkan:06}. Namely, they allow for
pseudoknots formed by exterior arcs as reported, for instance,
in Homo sapiens ACA27 snoRNA, see Figure~\ref{F:chy} and
Figure~\ref{F:dis}.

Let us next compare AP-structures and $0$-structures over two
backbones in more detail.  Recall that
an AP-structure, $J(R,S,I)$, is a graph such that
\begin{enumerate}
\item $R$, $S$ are secondary structures,
\item $I$ is a set of exterior arcs without external pseudoknots,
\item $J(R,S,I)$ contains no zig-zags.
\end{enumerate}
A tight AP-structure ($R(TS)$) is a substructure that cannot
be decomposed via block decomposition \citet{rip:09,Huang:10a}.
Accordingly, the shadow of a $R(TS)$ is connected and hence
irreducibile. $R(TS)$ and tight structures of $0$-structures
over two backbones are distinct concepts.
We have already observed that $0$-structures over two backbones
are not contained in the set of AP-structures.
Likewise, AP-structures are not contained in the set of $0$-structures over two
backbones, for example, consider a shadow of a $0$-structure over two backbones
which consist of $3<x$ distinct, irreducible shadows over two backbones
having genus $0$. According to Lemma.~\ref{L:genus_cal}, the genus of this
diagram is $x-1$. Drawing an interior arc covering the $R$-endpoints of
these $x$ shadows tightly, the resulting diagram is by construction a $R(TS)$ as in
Figure~\ref{F:dis}. As inserting a single arc changes the genus at
most by one, the diagram, $R(TS)$, has genus $\ge 1$, has an irreducible shadow
and is consequently not a $0$-structure over two backbones.
\begin{figure}[ht]
\centerline{\epsfig{file=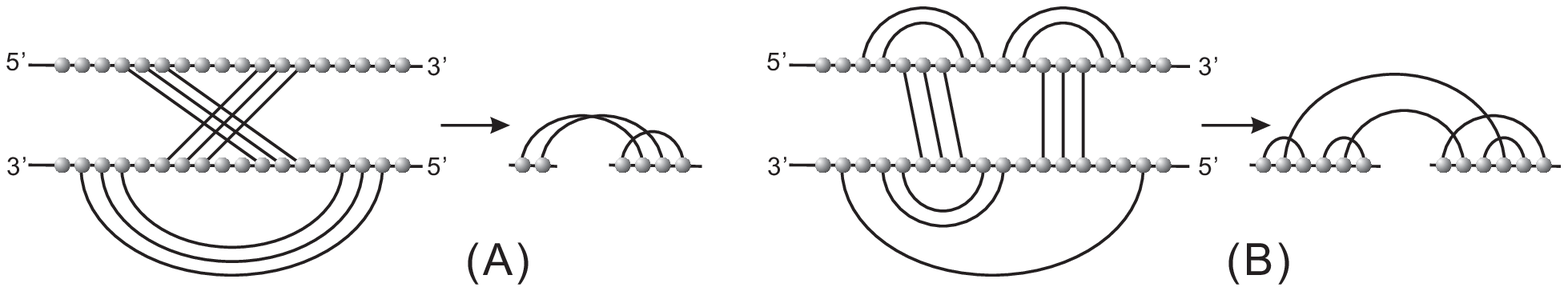,width=0.9\textwidth} \hskip8pt}
\caption{\small (A): a $0$-structure over two backbones that is {\it not} an
AP-structure; the crossing hybrid.
(B): an AP-structure that is {\it not} a $0$-structure over two backbones;
this structure contains an irreducible shadow over two backbones of
genus $1$.
}\label{F:dis}
\end{figure}

\end{document}